\documentclass[reqno, 12pt]{amsart}
\usepackage{amsmath, amssymb}
\usepackage[all]{xypic}
\newcommand{\eeq}{\end{equation}}

\usepackage{times}
\usepackage[T1]{fontenc}
\usepackage{mathrsfs}
\usepackage{epsfig}
\usepackage{color}

\newtheorem{thm}{Theorem}[section]

\newtheorem{prop}[thm]{Proposition}
\newtheorem{lem}[thm]{Lemma}
\newtheorem{cor}[thm]{Corollary}

\def\slfrac#1#2{\hbox{\kern.1em %
 \raise.5ex\hbox{\the\scriptfont0 #1}\kern-.11em %
 /\kern-.15em\lower.25ex\hbox{\the\scriptfont0 #2}}}

\title[Differential polynomial rings over a PI ring]{Differential polynomial rings over rings satisfying a polynomial identity}


\author{Jason P. Bell}
\thanks{The authors thank NSERC for its generous support.}
\address{University of Waterloo \\
Department of Pure Mathematics \\
Waterloo, Ontario \\
Canada  N2L 3G1\\}
\email{jpbell@uwaterloo.ca}

\author{Blake W. Madill}
\email{bmadill@uwaterloo.ca}

\author{Forte Shinko}
\email{fshinko@uwaterloo.ca}


\begin{document}
 
\begin{abstract} Let $R$ be a ring satisfying a polynomial identity and let $\delta$ be a derivation of $R$.  We show that if $N$ is the nil radical of $R$ then $\delta(N)\subseteq N$ and the Jacobson radical of $R[x;\delta]$ is equal to $N[x;\delta]$.  As a consequence, we have that if $R$ is locally nilpotent then $R[x;\delta]$ is locally nilpotent.  This affirmatively answers a question of Smoktunowicz and Ziembowski.  
 \end{abstract}
 

\maketitle

\section{Introduction}
Let $R$ be a ring (not necessarily unital) and let $\delta$ be a derivation of $R$.  We recall that the differential polynomial ring $R[x;\delta]$ is, as a set, given by all polynomials of the form
$a_n x^n+\cdots +a_1 x +a_0$ with $n\ge 0$, $a_0,\ldots ,a_n\in R$.  Multiplication is given by $xa=ax+\delta(a)$ for $a\in R$ and extending using associativity and linearity.

There has been a lot of interest in studying the Jacobson radical of the ring $R[x;\delta]$ \cite{Am, BG, FKM, Jor, SZ, TWC}.  In the case when $\delta$ is the zero derivation---that is when $R[x;\delta]=R[x]$---Amitsur \cite{Am} showed that the Jacobson radical, $J(R[x])$, of $R[x]$ is precisely $N[x]$ where $N$ is a nil ideal of $R$ given by $N=J(R[x])\cap R$.   At the opposite end of the spectrum, Ferrero, Kishimoto, and Motose \cite{FKM} showed that when $R$ is commutative then $J(R[x;\delta])\cap R$ is a nil ideal and $J(R[x;\delta])=(J(R[x;\delta])\cap R)[x;\delta]$.   It is still unknown whether $J(R[x;\delta])$ is equal to $(0)$ when $R$ has no nonzero nil ideals.  

A surprising recent development comes from the work of Smoktunowicz and Ziembowski \cite{SZ}.  We recall that a ring $R$ is locally nilpotent if every finitely generated subring of $R$ is a nilpotent ring.  Smoktunowicz and Ziembowski negatively answered a question of Sheshtakov \cite[Question 1.1]{SZ}, by constructing an example of a locally nilpotent ring $R$ such that $R[x;\delta]$ is not equal to its own Jacobson radical.   In addition to this, they asked \cite[p. 2]{SZ} whether Sheshtakov's question has an affirmative answer if one assumes, in addition, that $R$ satisfies a polynomial identity (PI ring, for short).  In this paper we show that this is indeed the case.

Our main result is the following theorem.
\begin{thm}
\label{NilRing}
Let $R$ be a locally nilpotent ring satisfying a polynomial identity and
let $\delta$ be a derivation of $R$.  
Then $R[x;\delta]$ is locally nilpotent.
\end{thm} 
In particular, this result shows that $R[x;\delta]$ is equal to its own Jacobson radical under the hypotheses from the statement of Theorem \ref{NilRing}.  This gives an affirmative answer to a question of Smoktunowicz and Ziembowski \cite[p. 2]{SZ}.  We note that the analogue of Theorem \ref{NilRing} need not hold if we form a skew polynomial extension of $R$ using an automorphism $\sigma$ instead of a derivation $\delta$.  For example, consider the ring $R=M/M^2$ where $M$ is the maximal ideal $(t_n\colon n\in \mathbb{Z})$ of $\mathbb{C}[t_n\colon n\in \mathbb{Z}]$ and let $\sigma$ be the automorphism of $R$ given by $\sigma(t_i+M^2)=t_{i+1}+M^2$.  Then $R$ is commutative and $R^2=0$ but $t_0x\in R[x;\sigma]$ is not nilpotent.

 As a corollary of Theorem \ref{NilRing}, we obtain---in the characteristic zero case---a result that can be thought of as an extension of a result of Ferrero, Kishimoto, and Motose \cite{FKM} to polynomial identity rings.  (Our result does not hold in the positive characteristic case, however; we give examples which show this.)  We recall that in a ring satisfying a polynomial identity, we have a two-sided nil ideal called the \emph{nil radical}.  This ideal is the sum of all right nil ideals \cite[Proposition 1.6.9 and Corollary 1.6.18]{Ro} and is locally nilpotent \cite{K}.  In general, it is unknown whether the sum of left nil ideals is again nil---this is the famous K\"othe conjecture.  
\begin{thm}
\label{cor: J}
Let $R$ be a unital polynomial identity algebra over a field of characteristic zero and let $\delta$ be a derivation of $R$.  Then if $N$ is the nil radical of $R$ then $\delta(N)\subseteq N$ and $J(R[x;\delta])=N[x;\delta]$.  In particular, $J(R[x;\delta])=(J(R[x;\delta])\cap R)[x;\delta]$.  
\end{thm}
We give examples that show that the containment $\delta(N)\subseteq N$ need not hold if the characteristic zero hypothesis is dropped; in particular, the equality $J(R[x;\delta])=N[x;\delta]$ need not hold without this hypothesis.

The outline of this paper is as follows.  In Section 2, we give some results from combinatorics on words.  In Section 3, we use these combinatorial results to prove Theorems \ref{NilRing} and \ref{cor: J}.  We note that throughout this paper, we have opted to work with rings that are not necessarily unital---the reason for this is that the question of Smoktunowicz and Ziembowski was asked for such rings.  On the other hand, we occasionally refer to results that sometimes implicitly assume that the involved rings are unital.  In practice, this does not create any issues: to a non-unital ring $R$, one can create an overring $S$ of $R$ with identity in which $R$ sits as a two-sided ideal and has the property that $S/R$ is a homomorphic image of $\mathbb{Z}$ (possibly $\mathbb{Z}$, itself).  By working with the ring $S$, one can generally apply any results stated for unital rings to $S$ and then show they are inherited by $R$.  Since this is generally straightforward, we make no mention of this other than here.
\section{Combinatorics on words}
In this section, we give some results on combinatorics on words that will be useful to us.  We begin by recalling some of the basic notions we will use.

For any set $A$, let $A^+$ denote the free semigroup on $A$.
We will refer to the elements of $A^+$ as \emph{words}.  \ 
For any $u\in A^+$, we let $u_i\in A$ denote the $i$-th letter of $u$.
A \emph{subword} of $u$ is a contiguous string of letters of $u$, possibly empty.  We say that a subword $v$ of $u$ is a \emph{prefix} if $u=vw$ for some word $w$, possibly empty; we say that $v$ is a \emph{suffix} if $u=wv$ for some, possibly empty, subword $w$ of $u$.  We will be interested in the case when $A=\mathbb{N}:=\{0,1,\ldots \}$.  

Let $S_n$ be the symmetric group on $n$ letters.
Define ${\rm weight}\colon\mathbb{N}^+\to \mathbb{N}$ as follows.  If $u\in \mathbb{N}^+$ is a word of length $n$ then we define
\[
  {\rm weight}(u)
  := \min\left\{\sum_{i=1}^n (n+1 - i) u_{\sigma(i)}\mid \sigma\in S_n\right\}.
\]
For a natural number $k$ and a word $u\in \mathbb{N}^+$ of length $n$, we say that $u$ is $k$-\emph{valid} if ${\rm weight}(u) \le k{n+1\choose 2}$.  Roughly speaking, this says that the average letter of $u$ is not too large compared to $k$.  Let ${\bf b}=(b_0,b_1,b_2,\ldots )$ be a sequence of natural numbers.
We say that $u\in\mathbb{N}^+$ is ${\bf b}$-bounded if for each $m\in\mathbb{N}$,
every subword of length $b_m$ contains at least one letter greater than $m$.

Let $(B,<)$ be a poset.
We place a partial order $\prec$ on $B^+$ as follows.
Let $u,v\in B^+$.
Then $u$ and $v$ are incomparable if one is a prefix of the other.
Otherwise, we compare them lexicographically using the order from $B$.  We will be most interested in this when $B$ is the natural numbers and we will make use of this induced order on $\mathbb{N}^+$.

We say that a finite sequence of words $\{v_i\}_{i=1}^d\subseteq B^+$ is $d$-\emph{decreasing}
if $$v_1 \succ v_2 \succ \cdots \succ v_d.$$
We say that a word $u\in B^+$ has a $d$-decreasing subword
if we can express $u = v w_1 w_2\cdots w_d x$ where $\{w_i\}_{i=1}^d$ is a $d$-decreasing subsequence.  We observe that every word trivially contains a $0$-decreasing subsequence.

\begin{prop}
\label{CoW}
Let ${\bf b}=(b_0,b_1,\ldots )$ be a sequence of natural numbers, let $d$ and $k$ be positive integers, and let $\varepsilon\in (0,1]$.
Then there exist natural number constants $M=M(d,{\bf b},k,\varepsilon)$ and $N=N(d,{\bf b},k,\varepsilon)$
such that if $u\in\mathbb{N}^+$ is a $k$-valid, ${\bf b}$-bounded word of length $n \ge N$,
then the subword of $u$ consisting of the last $\lfloor{\varepsilon n}\rfloor$ letters contains a $d$-decreasing subsequence $\{w_i\}_{i=1}^d$ where the first letter of $w_i$ is less than $M$ for $i=1,\ldots ,d$.
\end{prop}

\begin{proof}
We proceed by induction on $d$.
The case when $d = 0$ is vacuous and we may take $M=M(0,b,k,\varepsilon)=N(0,b,k,\varepsilon)=1$.

Suppose now that the proposition is true for all nonnegative integers $\le d$.
We take 
\begin{equation}
M_1 = M\left(d,{\bf b},k,\varepsilon/2\right)\qquad {\rm and}\qquad
N_1 = N\left(d,{\bf b},k,\varepsilon/2\right).
\end{equation}

We pick a positive integer $M_2$ satisfying:
\begin{enumerate}
\item[(i)] $M_2 > M_1$;
\item[(ii)] $M_2 > 8b_{M_1}^2 k\varepsilon^{-2}$.
\end{enumerate}
We have
$$
  M_2 {(\varepsilon n/2  -1)/b_{M_1}\choose 2}
  \sim M_2\varepsilon^2 b_{M_1}^{-2}n^2/8 \ge kn^2.$$
  It follows that there is a natural number $N_2>N_1$ such that whenever $n>N_2$ we have
\begin{equation}
\label{eq: MM}
 M_2 {(\varepsilon n/2  -1)/b_{M_1} \choose 2} > k{n+1\choose 2}.
\end{equation}

Let $u\in \mathbb{N}^+$ be a $k$-valid, ${\bf b}$-bounded word of length $n\ge N_2$.  We write 
$u = vwx$,
where $wx$ is of length $\lfloor{\varepsilon n}\rfloor$ and $x$ is of length $\lfloor{\varepsilon n/2}\rfloor$.
We decompose $w$ into subwords of length $b_{M_1}$ as follows.
We write $w = y_1\cdots y_j y_{j+1}$ where each of $y_1,\ldots ,y_j$ has length $b_{M_1}$
and $y_{j+1}$ has length less than $b_{M_1}$ (possibly zero).
By construction,
\begin{equation}
j = \left\lfloor{\frac{\lfloor{\varepsilon n}\rfloor - \lfloor{\varepsilon n/2} \rfloor}{b_{M_1}}}\right\rfloor.
\end{equation}
Since $y_i$ has length $b_{M_1}$ for $i\in \{1,\ldots ,j\}$, 
it must contain a letter $a_i$ with $a_i\ge M_1$.

We claim that there exists some $i\in \{1,\ldots ,j\}$ such that $a_i< M_2$.
To see this, suppose that this is not the case.  Then $u$ contains at least $j$ letters that are each at least $M_2$.
Since $u$ contains $j$ letters that are at least $M_2$, we have
$${\rm weight}(u)\ge j M_2 + (j-1) M_2 + \cdots +  M_2 = M_2{j+1\choose 2}\ge M_2 {(\varepsilon n/2  -1)/b_{M_1} \choose 2}.$$
But Equation (\ref{eq: MM}) gives that this contradicts the fact that $u$ is a $k$-valid word.  

We conclude that $w$ must contain a letter $a$ with $M_1 \le a < M_2$.
We write $w = bc$ where $c$ is a word whose first letter is $a$.

By the inductive hypothesis, we can write $x = p v_1 \cdots v_d q$
where $v_1\succ \cdots \succ v_d$ and the first letter of  $v_i$ is strictly less than $M_1$ for $i\in \{1,\ldots ,d\}$.
We then have that $wx = bcpv_1\cdots v_d q$, and by construction
$$cp \succ v_1 \succ \cdots \succ v_d$$ is a $(d+1)$-decreasing subsequence
where the first letter of each word in the sequence is less than $M_2$.  
The result now follows taking
$$N(d+1,{\bf b},k,\varepsilon) = N_2 \qquad {\rm and}\qquad M(d+1,{\bf b},k,\varepsilon) = M_2.$$
\end{proof}

\begin{cor}
\label{CoW2}
Let ${\bf b}=(b_0,b_1,\ldots )$ be a sequence of natural numbers and let $d$ and $k$ be positive integers.
Then there exists a natural number $N=N(d,b,k)$
such that if $u\in\mathbb{N}^+$ is a $k$-valid, ${\bf b}$-bounded word of length $n \ge N$,
then $u$ contains a $d$-decreasing subword.\end{cor}
\begin{proof}
We take $N(d,b,k) = N(d,b,k,1)$ from Corollary \ref{CoW}.
\end{proof}
\section{Proofs of Theorems \ref{NilRing}  and \ref{cor: J}}

In this section we prove our main results.  We begin with a simple lemma that will allow us to apply our combinatorial results from the preceding section.
\begin{lem}\label{DerLem}
Let $R$ be a ring, let $T=\{a_1,\ldots ,a_m\}$ be a finite subset of $R$, and let $\delta$ be a derivation of $R$.  If $n$ and $k$ are natural numbers and $p_1,\ldots ,p_{n+1}$ are nonnegative integers that are at most $k$ then the product $$a_{i_0}x^{p_1}a_{i_1}x^{p_2}\cdots a_{i_n}x^{p_{n+1}}$$ in $R[x;\delta]$ can be written as a $\mathbb{Z}$-linear combination of of elements of the form $$a_{i_0}\delta^{j_1}(a_{i_1})\delta^{j_2}(a_{i_2})\cdots \delta^{j_{n}}(a_{i_n})x^M,$$ where
$M$ is a nonnegative integer and $j_1j_2\cdots j_n\in  \mathbb{N}^+$ is $k$-valid. 
\end{lem}

\begin{proof}
Using the formula
\begin{equation}
x^d a = \sum_{j=0}^d {d\choose j} \delta^j(a) x^{d-j},
\end{equation}
for $d\ge 0$ and $a\in R$,
it is straightforward to see that
$$a_{i_1}x^{p_1}a_{i_2}x^{p_2}\cdots a_{i_n}x^{p_n}$$
can be expressed as a $\mathbb{Z}$-linear combination of elements of the form
$$a_{i_0}\delta^{j_1}(a_{i_1})\cdots \delta^{j_n}(a_{i_n}) x^{p_1+p_2+\cdots +p_n+p_{n+1}-j_1-\cdots -j_n}$$
where we have
$j_i \le p_1+\cdots +p_i-j_1-\cdots -j_{i-1}$ for $i=1,\ldots ,n$.  
In particular, we have
$$j_1+\cdots +j_i \le p_1+\cdots +p_i\le ki$$ for $i=1,\ldots ,n$.  
Summing over all $i$ then gives
$$\sum_{i=1}^n (n+1-i) j_i \le \sum_{i=1}^n ki = k {n+1\choose 2}.$$
Thus the word $j_1j_2\cdots j_n\in \mathbb{N}^+$ is necessarily $k$-valid.  The result follows.
\end{proof}

\begin{proof}[Proof of Theorem \ref{NilRing}]
Let $S=\{p_1(x),\ldots ,p_m(x)\}$ be a finite subset of $R[x;\delta]$.  We wish to show that there is a natural number $N=N(S)$ such that $S^{N+1}=0$; i.e., every product of $N+1$ elements of $S$ is equal to zero.   Then there is a finite subset $T=\{a_1,\ldots ,a_t\}$ of $R$ and a natural number $k$ such that
$S\subseteq T+Tx+\cdots +Tx^k$.   Let $S_0=T\cup Tx\cup \cdots \cup Tx^k$.  Then every element of $S^n$ can be expressed as a sum of elements of the form $S_0^n$ and hence it is sufficient to show that there is a natural number $N$ such that $S_0^{N+1}=0$.
  
To show that $S_0^{N+1}=0$, it is enough to show that 
$$Tx^{p_1}Tx^{p_2}\cdots Tx^{p_{N+1}}=0$$ for every sequence $(p_1,\ldots ,p_{N+1})\in \{0,\ldots ,k\}^{N+1}$.  For each $n\ge 0$, we let
$T_n=T\cup \delta(T)\cup \cdots \cup \delta^n(T)\subseteq R$.  Then since $R$ is locally nilpotent, there exists a natural number $b_n$ such that $T_n^{b_n}=0$.  We let ${\bf b}=(b_0,b_1,b_2,\ldots )$.  

By Lemma \ref{DerLem}, if $(p_1,\ldots ,p_{N+1})\in \{0,\ldots ,k\}^{N+1}$ then we have that 
$$Tx^{p_1}Tx^{p_2}\cdots Tx^{p_{N+1}}$$ can be written as a $\mathbb{Z}$-linear combination of elements of the form
$$a_{i_0}\delta^{j_1}(a_{i_1})\cdots \delta^{j_{N}}(a_{i_{N}}) x^M$$ where $M$ is a nonnegative integer and 
$j_1j_2\cdots j_N$ is a word that is $k$-valid.  Moreover, whenever $j_1j_2\cdots j_N$ is not ${\bf b}$-bounded we trivially have $$a_{i_0}\delta^{j_1}(a_{i_1})\cdots \delta^{j_{N}}(a_{i_{N}}) =0,$$ since it necessarily contains a factor from $T_n^{b_n}$ for some $n\ge 0$.  In particular, it is sufficient to show that there is some natural number $N$ such that all elements of $R$ of the form
$$a_{i_0}\delta^{j_1}(a_{i_1})\cdots \delta^{j_{N}}(a_{i_{N}}),$$ with $j_1j_2j_3\cdots j_N\in \mathbb{N}^+$ a $k$-valid and ${\bf b}$-bounded word, are zero.

Let $d$ be the PI degree of $R$ and let $N=N(d,{\bf b},k)$ be as in the statement of Corollary \ref{CoW2}.   We claim that whenever $j_1j_2j_3\cdots j_N$ a $k$-valid and ${\bf b}$-bounded word we have $a_{i_0}\delta^{j_1}(a_{i_1})\cdots \delta^{j_{N}}(a_{i_{N}})=0$.  To see this, suppose towards a contradiction that this is not the case and let $j_1\cdots j_N$ be the lexicographically smallest (i.e., the smallest word with respect to $\prec$) $k$-valid and ${\bf b}$-bounded word of length $N$ such that 
there exists $(i_0,\ldots ,i_N)\in \{1,\ldots ,m\}^{N+1}$ such that $a_{i_0}\delta^{j_1}(a_{i_1})\cdots \delta^{j_{N}}(a_{i_{N}})$ is nonzero.  

Given a subword $y=j_sj_{s+1}\cdots j_{s+r}$ of $j_1\cdots j_N$, we 
define $$f(y)=\delta^{j_s}(a_{i_s})\cdots \delta^{j_{s+r}}(a_{i_{s+r}})\in R.$$

By Corollary \ref{CoW2}, we can write
$j_1\cdots j_N=uw_1w_2\cdots w_d v$ with $$w_1\succ w_2 \succ \cdots \succ w_d.$$   Furthermore, we have that $R$ satisfies a homogeneous multilinear polynomial identity of degree $d$ \cite[Proposition 13.1.9]{MR}:
\[
  X_1\cdots X_d
  = \sum_{\substack{\sigma\in S_d \\ \sigma\neq {\rm id}}} c_\sigma X_{\sigma(1)}\cdots X_{\sigma(d)}
\]
with $c_\sigma \in\mathbb{Z}$ for each $\sigma\in S_d\setminus \{{\rm id} \}$.
Taking $X_i=f(w_i)$ for $i=1,\ldots, d$ we see that
$$f(w_1)f(w_2)\cdots f(w_d) \ =  \ \sum_{\substack{\sigma\in S_d \\ \sigma\neq{\rm id}}} c_\sigma f(w_{\sigma(1)})\cdots f(w_{\sigma(d)}).$$
Hence 
\begin{equation}
\label{eq: sigma}
a_{i_0}\delta^{j_1}(a_{i_1})\cdots \delta^{j_{N}}(a_{i_{N}}) = \sum_{\substack{\sigma\in S_d \\ \sigma\neq {\rm id}}} c_\sigma a_{i_0}f(u)f(w_{\sigma(1)})\cdots f(w_{\sigma(d)})f(v).
\end{equation}
By construction, for $\sigma\in S_d$ with $\sigma\neq {\rm id}$ we have
$a_{i_0}f(u)f(w_{\sigma(1)})\cdots f(w_{\sigma(d)})f(v)$ is an element of the form
$a_{i_0} \delta^{j_{\tau(1)}}(a_{i_{\tau(1)}})\cdots \delta^{j_{\tau(N)}}(a_{i_{\tau(N)}})$ with 
$\tau\in S_N$ such that
$j_{\tau(1)}\cdots j_{\tau(N)}$ is lexicographically less than $j_1\cdots j_N$.
We note that, by definition, permutations of $k$-valid words are again $k$-valid.  Thus if we also have that if $j_{\tau(1)}\cdots j_{\tau(N)}$ is ${\bf b}$-bounded then we must have $a_{i_0} \delta^{j_{\tau(1)}}(a_{i_{\tau(1)}})\cdots \delta^{j_{\tau(N)}}(a_{i_{\tau(N)}})=0$ by minimality of $j_1\cdots j_N$.  On the other hand, if 
$j_{\tau(1)}\cdots j_{\tau(N)}$ is not ${\bf b}$-bounded then $a_{i_0} \delta^{j_{\tau(1)}}(a_{i_{\tau(1)}})\cdots \delta^{j_{\tau(N)}}(a_{i_{\tau(N)}})$ contains a factor that lies in $T_n^{b_n}$ for some $n\ge 0$ and hence it is zero.  Thus we have shown that in either case we have $a_{i_0} \delta^{j_{\tau(1)}}(a_{i_{\tau(1)}})\cdots \delta^{j_{\tau(N)}}(a_{i_{\tau(N)}})$ is zero for all applicable $\tau$, and so
from Equation (\ref{eq: sigma}) we see $$a_{i_0}\delta^{j_1}(a_{i_1})\cdots \delta^{j_{N}}(a_{i_{N}})=0,$$
a contradiction.  It follows that $S^{N+1}=0$.
\end{proof}
We now deduce Theorem \ref{cor: J} from Theorem \ref{NilRing}.  
\begin{proof}[Proof of Theorem \ref{cor: J}]
Since $R$ is PI, the sum of all nil right ideals is a nil two-sided locally nilpotent ideal $N$, which is called the nil radical of $R$ (see Rowen \cite[Proposition 1.6.9 and Corollary 1.6.18]{Ro} and Kaplansky \cite{K}).  We claim that $\delta(N)\subseteq N$.  To see this, suppose to the contrary that $\delta(N)\not\subseteq N$.  Then there is some $a\in N$ such that $\delta(a)\not\in N$.  Then either $\delta(a)$ is not nilpotent or there is some $r\in R$ such that $\delta(a)r$ is not nilpotent.  In the latter case, we have $\delta(ar)=\delta(a)r+a\delta(r)\equiv \delta(a)r~(\bmod ~N)$ and so if $\delta(a)r$ is not nilpotent then neither is $\delta(ar)$.  Hence in either case we see we can find an element $b\in N$ such that $\delta(b)$ is not nilpotent.  By assumption, there is some $n\ge 2$ such that $b^n=0$.  It is straightforward to show that there exist nonnegative integers $c_{j_1,\ldots ,j_n}$ such that
\begin{equation}
\label{eq: b}
0=\delta^n(b^n) = \sum_{j_1+\cdots +j_n=n} c_{j_1,\ldots ,j_n} \delta^{j_1}(b)\cdots \delta^{j_n}(b).
\end{equation}
Moreover, we have that $c_{1,1,\ldots ,1}\ge 1$ and hence is nonzero in $R$ since $R$ is a unital algebra over a field of characteristic zero.  
Observe that if $j_1+\cdots +j_n=n$ and $(j_1,\ldots ,j_n)\neq (1,1,\ldots ,1)$ then $j_i=0$ for some $i$ and so $\delta^{j_1}(b)\cdots \delta^{j_n}(b)\in RbR\subseteq N$.  Hence
Equation (\ref{eq: b}) gives
$\delta(b)^n \in N$.  Since $N$ is a nil ideal, it follows that $\delta(b)$ is nilpotent, a contradiction.   Thus $\delta(N)\subseteq N$.  Notice that $N$ is locally nilpotent and so the subring $N[x;\delta]$ of $R[x;\delta]$ is a locally nilpotent ideal of $R[x;\delta]$ by Theorem \ref{NilRing}.   It follows that $N[x;\delta]\subseteq J(R[x;\delta])$.  

To show that $J(R[x;\delta])$ is equal to $N[x;\delta]$, it suffices to show that $(R/N)[x;\delta]$ has zero Jacobson radical.  A result of Tsai, Wu, and Chuang \cite{TWC} gives that if $S$ is a PI ring with zero nil radical then the Jacobson radical of $S[x;\delta]$ is zero.  The result now follows.
\end{proof}
We note that the characteristic zero hypothesis is essential in Theorem \ref{cor: J}.  For example, if $p>0$ is a prime number and we let $R=\mathbb{F}_p[T]/(T^p)$ and let $t$ denote the image of $T$ in $R$, then $R$ has a unique derivation satisfying $\delta(t)=1$.  It is clear that $R$ is commutative (and hence PI) and that the nil radical of $R$ is not closed under application of $\delta$ (since $t$ is in the nil radical and $\delta(t)=1$).  In fact, the only proper ideal of $R$ closed under application of $\delta$ is easily seen to be $(0)$.  In this case, the result of Ferrero, Kishimoto, and Motose \cite{FKM} gives that $S:=J(R[x;\delta])\cap R$ is a nil ideal of $R$ that is closed under $\delta$ and thus we see that $S$ is necessarily zero.  They also show that $J(R[x;\delta])=S[x;\delta]$ and so the Jacobson radical is zero in this case.

\bibliographystyle{plain}
\end{document}